\documentclass[12pt]{amsart}

\usepackage{amsmath,amssymb}

\begin{document}

\newcommand\Lp{{\mathcal L}_+}
\newcommand\F{\mathcal F}
\newcommand\D{\mathcal D}
\newcommand\U{\mathcal U}
\newcommand\V{\Theta}
\newcommand\M{\mathcal M}
\newcommand\NN{\mathcal N}
\newcommand\J{\mathcal J}
\newcommand\T{\mathcal T}
\newcommand\JV{\mathcal{J}\Theta}
\newcommand\AV{\mathcal{AV}}
\newcommand\OO{\mathcal O}
\newcommand\RR{\mathcal R}
\newcommand\A{\OO}
\newcommand\hh{\widehat{\#}}
\newcommand\wo{\widehat{\otimes}}
\newcommand\C{\mathbb C}
\newcommand\R{\mathcal R}
\newcommand\Z{\mathbb Z}
\newcommand\N{\mathbb N}
\newcommand\Proj{\mathbb P}
\newcommand\Aff{\mathbb A}
\newcommand\Aut{\text{Aut\,}}
\newcommand\Diff{\text{Diff\,}}
\newcommand\Der{\text{Der\,}}
\newcommand\Ind{\text{Ind\,}}
\newcommand\End{\text{End\,}}
\newcommand\Span{\text{Span\,}}
\newcommand\Hom{\text{Hom}}
\newcommand\ad{\text{ad}}
\newcommand\tr{\text{tr}}
\newcommand{\dd}[1]{\frac{d}{d#1}}
\newcommand{\pd}[1]{\frac{\partial}{\partial#1}}
\newcommand{\pp}[2]{\frac{\partial#1}{\partial#2}}
\newcommand\del{\partial}
\newcommand\g{\mathfrak g}
\newcommand\m{\mathfrak m}
\newcommand\gl{\mathfrak{gl}}
\newcommand\hL{\widehat{L}}
\newcommand\tA{\widetilde{A}}
\newcommand\tV{\widetilde{V}}
\newcommand\Uw{U_\text{weak}}
\newcommand\Us{U_\text{strong}}
\newcommand\eps{\epsilon}

\newtheorem{theorem}{Theorem}
\newtheorem{proposition}[theorem]{Proposition}
\newtheorem{corollary}[theorem]{Corollary}
\newtheorem{lemma}[theorem]{Lemma}
\newtheorem{definition}[theorem]{Definition}
\newtheorem{remark}[theorem]{Remark}
\newtheorem{conjecture}[theorem]{Conjecture}
\newtheorem{example}[theorem]{Example}

\hyphenation{sub-algebra}

\

\title
[Sheaves of $AV$-modules]
{Sheaves of $AV$-modules on quasi-projective varieties}
\author{Yuly Billig}
\address{School of Mathematics and Statistics, Carleton University, Ottawa, Canada}
\email{billig@math.carleton.ca}
\address{Academia Sinica, Taipei, Taiwan}
\author{Emile Bouaziz}
\email{emile.g.bouaziz@gmail.com}
\address{}
\email{}
\let\thefootnote\relax

\maketitle

\begin{abstract}
We study sheaves of modules for the Lie algebra of vector fields with the action of the algebra of functions, compatible via the Leibniz rule. A crucial role in this theory is played by the virtual jets of vector fields -- jets that evaluate to a zero vector field under the anchor map. Virtual jets of vector fields form a vector bundle $\Lp$ whose fiber is Lie algebra $\hL_+$ of vanishing at zero derivations of power series. We show that a sheaf of $AV$-modules is characterized by two ingredients -- it is a module for $\Lp$ 
and an $\Lp$-charged $D$-module.

For each rational finite-dimensional representation of $\hL_+$, we construct a bundle of jet $AV$-modules. We also show that Rudakov modules may be realized as tensor products of jet modules with a $D$-module of delta functions.   
\end{abstract}

\section{Introduction}

The theory of $D$-modules (modules over the algebra of differential operators) has been very successful. The technique of $D$-modules was used in the proof of Kazhdan-Luzstig conjecture \cite{BB, BK} and in the formulation of the geometric Langlands correspondence \cite{F, GR}. $D$-modules may be viewed as modules over the Lie algebra of vector fields, together with the action of the algebra of functions, satisfying the Leibniz rule and subject to the condition that the composition $f \circ \eta$ of the actions by function $f$ and vector field $\eta$ coincides with the action by vector field $f \eta$.

There are many examples from geometry where we have the actions of both vector fields and functions, where the above composition property fails. It does not hold, for example, for the adjoint action of vector fields. An $AV$-module is a generalization of a $D$-module, where we drop the composition axiom. We say that $M$ is an $AV$-module if $M$ is a module for Lie algebra $V$ of vector fields and a module for the commutative algebra $A$ of functions, with two actions compatible via the Leibniz rule.

A general theory of $AV$-modules on a smooth affine algebraic variety was developed in \cite{BF, BFN, BNZ, BIN, BR}. The machinery of $AV$-modules was indispensable in establishing classifications of simple weight modules for the Lie algebras of vector fields on a torus \cite{BF}, and an affine space \cite{XL, GS}. 

The category of $AV$-modules may be presented as modules over an associative algebra $A \# U(V)$, which is the weak enveloping algebra of Lie-Rinehart pair $(A, V)$.

The goal of the present paper is to extend the theory of $AV$-modules to the case of a smooth quasiprojective variety $X$ of dimension $n$, which requires working in a sheaf-theoretic setting. It was seen in \cite{BI} that sheafification requires taking a completion $A \widehat{\#} U(V)$ of the associative algebra $A \# U(V)$. This completion is the strong enveloping algebra of the Lie-Rinehart pair $(A, A\widehat{\#} V)$,
where  $A\widehat{\#} V$ is the Lie algebra of $\infty$-jets of vector fields. This yields a quasicoherent sheaf $\AV$ of associative algebras.

Locally, in an \'etale chart, we have an isomorphism \cite{XL, BIN, BI}:
$$A \widehat{\#} U(V) \cong D \otimes U(\hL_+),$$
where $D$ is the algebra of differential operators and $\hL_+$ is the Lie algebra of derivations of the algebra of power series $K[[X_1, \ldots, X_n]]$ which vanish at $0$.
In fact, $\AV$ contains as a subsheaf a bundle $\Lp$ of Lie algebras with $\hL_+$ as its fiber. This is a bundle of virtual jets of vector fields, that is, jets that evaluate to a zero vector field under the anchor map. In contrast to this, locally defined algebras of differential operators do not glue into a subsheaf in $\AV$.

Our main result, Theorem \ref{main}, states that a sheaf of $AV$-modules is defined by two ingredients -- it is a sheaf of modules for $\Lp$, and also an $\Lp$-charged $D$-module (see Definition \ref{deform}).

In Section \ref{shjetmod} we construct an important class of sheaves of jet modules. For this, we generalize the notion of the jacobian of a change of coordinates. First, we show that $\infty$-jets of functions is a bundle of commutative algebras with fiber $K[[X_1, \ldots, X_n]]$. The group of automorphisms of $K[[X_1, \ldots, X_n]]$ is an infinite-dimensional algebraic group scheme 
with Lie algebra $\hL_+$. We show that a change of coordinates transformation on an intersection $U_1 \cap U_2$ of two \'etale charts gives rise to an element of $\Aut K[[X_1, \ldots, X_n]]$ with coefficients in $\OO(U_1 \cap U_2)$. If we factor out the terms in $ K[[X_1, \ldots, X_n]]$ of degrees greater than 1, this will reduce to the jacobian of the change of coordinates.

We use this generalization of the jacobian to construct a functor from the category of rational finite-dimensional $\hL_+$-modules to the category of sheaves of $AV$-modules. The corresponding jet module is a vector bundle with the given $\hL_+$-module as its fiber.

In 1974, Rudakov introduced a class of modules for the Lie algebra of vector fields on an affine space. Generalizations of Rudakov modules supported at a non-singular point of an arbitrary affine variety were constructed in \cite{BFN}. In Section \ref{Rud}, we give a realization of Rudakov modules as tensor products of jet modules with the $D$-module of delta functions supported at a given point.

In Section \ref{Holo}, we define holonomic sheaves of $AV$-modules. We state a conjecture that every holonomic $AV$-module is differentiable, that is, there exists $N \geq 1$ for which the module is annihilated by all elements of $A \# V$ of the form
$$\sum_{k=0}^N (-1)^k {N \choose k} f^k \# f^{N-k} \eta$$
with $f \in A$, $\eta \in V$. In \cite{BR}, the authors proved that $AV$-modules that are finitely generated over $A$, are differentiable. This result is a special case of this conjecture. Rudakov modules provide another class of holonomic modules for which the conjecture holds (Lemma \ref{Rdiff}).

As an illustration of our methods, in the final section of the paper, we construct two families of rank 2 bundles of $AV$-modules on $\Proj^1$.

\

{\bf Acknowledgements:} The authors benefited from helpful conversations with Colin Ingalls and Henrique Rocha.
Y.B. gratefully acknowledges support with a Discovery grant from the Natural Sciences and Engineering Research Council of Canada.

\section{$AV$-modules}

Let $K$ be an algebraically closed field of characteristic 0.
Let $A$ be the commutative algebra of functions on a smooth irreducible affine algebraic variety $X$, and let $V = \Der (A)$ be the Lie algebra of derivations on $A$. Finally, let $D$ be the associative algebra of differential operators, defined as the subalgebra in $\End_{K} (A)$ generated by $A$ (acting on itself by multiplication) and $V$ (acting on $A$ by derivations). 

An $AV$-module $M$ is a module for both Lie algebra $V$ and for commutative unital algebra $A$, with the two actions compatible via the Leibniz rule:
$$\eta (f m) = \eta(f) m + f (\eta m), \ \ \text{for \ } \eta \in V, \, f \in A, \, m \in M.$$

The commutative algebra $A$ and the Lie algebra $V$ themselves are naturally $AV$-modules, with the former even a $D$-module.

The category of $AV$-modules has a tensor product $M \otimes_A N$ and internal mapping spaces $\operatorname{Map}(M,N)$ adjoint to the tensor product in the usual sense that $$\operatorname{Hom}_{AV}(M\otimes_{A} N,L)=\operatorname{Hom}_{AV}(M,\operatorname{Map}(N,L)).$$ $\operatorname{Map}(M,N)$ is constructed as the module $\operatorname{Hom}_{A}(M,N)$ of all $A$-linear homomorphisms from $M$ to $N$, with an evident action of $V$. The unit of the tensor structure is the $AV$ module $A$, and so in particular the above adjunction induces an isomorphism $$\operatorname{Hom}_{AV}(A,\operatorname{Map}(M,N))=\operatorname{Hom}_{AV}(M,N).$$  We record in particular the \emph{dual} $AV$-module $\operatorname{Map}(M,A)$, which we recall has underlying $A$-module $\operatorname{Hom}_{A}(M,A)$.  

Taking the dual module of $V$, we construct the module of differential 1-forms: $\Omega^1 = \Hom_A (V, A)$. By taking a tensor product of $m$ copies of $V$ with $k$ copies of $\Omega^1$, we construct $AV$-modules of $(m, k)$-tensors.

There exists an associative algebra that controls the category of $AV$-modules. This algebra is the smash product $A \# U(V)$ of the universal enveloping algebra $U(V)$, viewed as a Hopf algebra, with its module $A$. As a vector space, it is the space 
$A \otimes_K U(V)$, where the commutation relations between the elements of $A$ and $V$ are given by the Leibniz rule:
$\eta \cdot f = \eta(f) + f \cdot \eta$.

There is a natural surjective homomorphism of associative algebras $A \# U(V) \rightarrow D$. This implies that every $D$-module is automatically an $AV$-module. $D$-modules are precisely $AV$-modules with an additional axiom
$$f\, (\eta\, m) = (f \eta)\, m.$$

While $A$ is a $D$-module, the $AV$-module $V$ does not have a natural $D$-module structure, since the Lie bracket in $V$ is not $A$-linear and so the above axiom does not hold in $V$.

The pair of algebras $(A, V)$ is a Lie-Rinehart pair. Let us recall the definition:

\begin{definition}
A pair $(\tA, \tV)$ is a Lie-Rinehart pair if $\tA$ is a unital commutative associative algebra, $\tV$ is a Lie algebra, $\tV$ is an $\tA$-module, $\tV$ acts on $\tA$ by derivations, and the Lie bracket in $\tV$ satisfies
$$[\mu, f \eta] = \mu(f) \eta + f [\mu, \eta], \ \ \ \text{for \ } f \in \tA, \ \mu, \eta \in \tV.$$
\end{definition}

A Lie-Rinehart pair $(\tA, \tV)$ has two enveloping algebras: {\it the weak enveloping algebra}
$$\Uw (\tA, \tV) = \tA \# U(\tV),$$
and  {\it the strong enveloping algebra}
$$\Us (\tA, \tV) = \Uw (\tA, \tV) / \left< f \# \eta - 1 \# f \eta \right>.$$

When $A$ is the algebra of functions on a smooth irreducible affine variety, and $V = \Der(A)$,  $\Us (A, V)$ is isomorphic to the algebra $D$ of differential operators on $X$.

It is an important fact that subspace $A \# V$ is a Lie subalgebra in $ A \# U(V)$ with Lie bracket
$$[ f \# \eta, g \# \mu] = f \eta(g) \# \mu - g \mu(f) \# \eta + fg \# [\eta, \mu].$$ 

Note that $(A, A\# V)$ is also a Lie-Rinehart pair, and
$$\Us (A, A\# V) \cong \Uw (A, V).$$ 

\begin{definition}
An $AV$-module $M$ is called differentiable if there exists $N \geq 1$ such that the following elements of $A \# V$
$$ \sum_{k=0}^N (-1)^k {N \choose k} f^k \# f^{N-k} \eta$$
annihilate $M$ for all $f \in A$ and $\eta \in V$.
\end{definition}

Note that an $AV$-module is differentiable with $N=1$ precisely when it is a $D$-module.

For an $AV$-module to be $N$-differentiable is equivalent to the action of $V$ being a differential operator of order at most $N$, in the sense of Grothendieck \cite{G}.
It is easy to see that the subcategory of $N$-differentiable $AV$-modules is closed under tensor products and duality.

\section{Group of automorphisms of the algebra of power series and its representations}

In this section, we discuss the group of automorphisms of the algebra $K[[X_1, \ldots, X_n]]$ of power series, and its finite-dimensional representations.

First of all, we point out that every automorphism of $K[[X_1, \ldots, X_n]]$ is continuous in the power series topology. To see this, let $\m = \left< X_1, \ldots, X_n \right>$ be the (unique) maximal ideal in $K[[X_1, \ldots, X_n]]$. Since this is the unique maximal ideal, every automorphism preserves $\m$, and hence also preserves all powers of $\m$. This implies continuity.

Continuity of automorphisms implies that every automorphism $F$ of $K[[X_1, \ldots, X_n]]$ is determined by the images of $X_1, \ldots, X_n$, and may be written as $X_i \mapsto F_i$, $i=1,\ldots, n$, where
$$F_i = \sum_{s \in \Z_+^n \backslash \{ 0 \} } \frac{1}{s!} \,A_{i, s} X^s, {\hskip 1cm} {\text with \ } A_{i,s} \in K,$$
where the matrix $(A_{i, \epsilon_j})$ is invertible. Here $\epsilon_j$ is an element of $\Z_+^n$ with a single 1 in position $j$.

This yields an infinite-dimensional group scheme with the algebra of functions $K[A_{i,s} \, | \, i=1, \ldots, n, \, s \in \Z_+^n \backslash \{ 0 \} ]_{(\det)}$, localized at $\det = \det(A_{i, \epsilon_j})$.

The algebra of functions on a group has a Hopf algebra structure, and for $\Aut K[[X_1, \ldots, X_n]]$  the coproduct is given on the generators as follows:
$$\Delta(A_{i,s}) = \sum_{r=1}^{|s|} \sum_{1 \leq j_1, \ldots, j_r \leq n} \ \sum_{p\in P_r (s)}
A_{i, \epsilon_{j_1} + \ldots + \epsilon_{j_r}} \otimes (A_{j_1, p_1} \cdot \ldots \cdot A_{j_r, p_r}),$$
$$\Delta(\det{}^{-1}) = \det{}^{-1} \otimes \det{}^{-1}.$$  
Here for $s \in \Z^n_+$ we use notiations $|s| = s_1 + \ldots + s_n$, $s! = s_1! \ldots s_n!$, etc,
and $P_r (s)$ is a set of partitions of $s$ into $r$ parts, $s = p_1 + \ldots + p_r$.

We point out the peculiar property of this coproduct that it is linear in the first tensor factor, and non-linear in the second.

For example, 
$$\Delta(A_{i, \epsilon_a + \epsilon_b}) = A_{i, \epsilon_k} \otimes A_{k, \epsilon_a + \epsilon_b}
+  A_{i, \epsilon_k + \epsilon_\ell} \otimes A_{k, \epsilon_a} A_{\ell, \epsilon_b}.$$
Here, and throughout the paper we use Einstein's notations, with summation over repeated indices.

The above example of the coproduct is a reflection of the chain rule:
$$\frac{\del^2}{\del x_a \del x_b} F_i ( G (x)) = \frac{\del F_i}{\del x_k} (G(x)) \frac{\del^2 G_k}{\del x_a \del x_b}
+ \frac{\del^2 F_i}{\del x_k \del x_\ell} (G(x)) \frac{\del G_k}{\del x_a} \frac{\del G_\ell}{\del x_b} .$$

The Lie algebra of the group $\Aut K[[X_1, \ldots, X_n]]$ is a (proper) subalgebra in the Lie algebra of derivations of 
$K[[X_1, \ldots, X_n]]$. An argument, similar to one given above, shows that every derivation of the algebra of power series is continuous. Thus
$$\Der K[[X_1, \ldots, X_n]] = \mathop\oplus\limits_{i=1}^n  K[[X_1, \ldots, X_n]] \frac{\del}{\del X_i} .$$
The Lie algebra of the group $\Aut K[[X_1, \ldots, X_n]]$ is 
$$\hL_+ = \mathop\oplus\limits_{i=1}^n \m \frac{\del}{\del X_i}.$$

Derivations $\frac{\del}{\del X_i}$ do not belong to the Lie algebra of the group of automorphisms, since they correspond to the shifts $X_i \mapsto X_i + a$, which are not automorphisms of the algebra of power series.

Let us discuss Lie correspondence between the group of automorphisms and $\hL_+$. The group 
$\Aut K[[X_1, \ldots, X_n]]$ is a semidirect product of $GL_n$ (linear automorphisms) and the normal subgroup $\NN$ of
automorphisms $X_i \mapsto X_i + \text{\it higher order terms}$, for all $i = 1, \ldots, n$. Subgroup $\NN$ is pronilpotent -- it has a descending chain of normal subgroups with finite-dimensional nilpotent quotients.

Using coordinates $\{ A_{i,s} \}$, we obtain the following realizations:
$$\\Aut K[[X_1, \ldots, X_n]] = GL_n \ltimes \prod\limits_{r=1}^\infty V \otimes S^r(V^*),$$
$$\NN = \{I\} \times \prod\limits_{r=1}^\infty V \otimes S^r(V^*),$$
where $V = K^n$.

Likewise, Lie algebra $\hL_+$ is the direct sum of its subalgebra $L_0 \cong gl_n$, spanned by $\left\{ X_j  \frac{\del}{\del X_i} \right\}$, and a pronilpotent ideal $\m \hL_+$.

The exponential map $\exp: \ \m \hL_+ \rightarrow \NN$ associates to a derivation $\eta$ an automorphism, sending $X_i$ to
$\exp(\eta) X_i$. It is easy to see that the exponential map between $\m\hL_+$ and $\NN$ is bijective.

\begin{example}
$$\exp\left(\alpha X^2 \frac{d}{dX}\right) g(X) = g\left( \frac{X}{1 - \alpha X} \right).$$
\end{example}
 
 Finally, let us consider finite-dimensional representations of $\hL_+$ and of the group of automorphisms of the algebra of power series.

\begin{lemma}
[cf. \cite{B}]
\label{nilp}
Let $W$ be a finite-dimensional module for $\hL_+$. Then there exists $N \in \N$, depending on $\dim W$, such that 
$\m^N \hL_+$ annihilates $W$.
\end{lemma}
\begin{proof}
Let us outline the sketch of the proof. A version of this Lemma for the algebra of polynomials was given in \cite{B}.

Consider the action on $W$ by derivation $X_1 \frac{\del}{\del X_1} + \ldots + X_n \frac{\del}{\del X_n}$. The space $W$ may be decomposed into a direct sum of the generalized eigenspaces with respect to this operator. Denote by $S_k$ the span of monomials in $X_1, \ldots, X_n$ of total degree $k$. Then $S_k \frac{\del}{\del X_i}$ increases the eigenvalue by $k-1$. 
Since the total number of the generalized eigenspaces is finite, we conclude that for some $N \in \N$,
$S_k \frac{\del}{\del X_i}$ annihilates $W$ for all $k \geq N$ and for all $i=1,\ldots, n$.

It is easy to check that for $N > 1$
$$ \left[ S_N  \frac{\del}{\del X_i} \oplus S_{N+1}  \frac{\del}{\del X_i}, \, \m \frac{\del}{\del X_i} \right]
= \m^N \frac{\del}{\del X_i}.$$
This implies the claim of the Lemma.
\end{proof}

Let $W$ be a finite-dimensional $\hL_+$-module. We will call $W$ rational if the action of subalgebra $L_0$ integrates to a rational $GL_n$-module.

\begin{lemma}
\label{Liecor}
Let $(W, \rho)$ be a finite-dimensional rational  $\hL_+$-module.
Then $W$ admits the action of the group $\Aut  K[[X_1, \ldots, X_n]]$, compatible with the exponential map. 
\end{lemma}
\begin{proof}
We have seen in the previous Lemma that there exists $N$ such that $\m^N \hL_+$ annihilates $W$. Thus $W$ is a module for a finite-dimensional Lie algebra $\hL_+ / \m^N \hL_+$, which is a semidirect product of $gl_n$ with a nilpotent ideal
$\m \hL_+ / \m^N \hL_+$. 
It follows from the proof of Lemma \ref{nilp} that $\m \hL_+ / \m^N \hL_+$ acts on $W$ by nilpotent operators.
Using the exponential map, we can define the action on $W$ of the nilpotent quotient group
$$\overline{\NN} = \{I\} \times \prod\limits_{r=1}^N V \otimes S^r(V^*).$$ 
Now we have two algebraic groups, $GL_n$ and $\overline{\NN}$, acting rationally on $W$.
Their Lie algebras span the Lie algebra $\hL_+ / \m^N \hL_+$, which also acts on $W$. By Theorem 7.6 in \cite{Bo},
these actions extend to the action on $W$ of an algebraic group with Lie algebra $\hL_+ / \m^N \hL_+$.
Let us show that this algebraic group is 
$$GL_n \ltimes \prod\limits_{r=1}^N V \otimes S^r(V^*).$$ 
Group structure in this group is determined by the group structures of $GL_n$, $\overline{\NN}$, and by the conjugation action of $GL_n$ on $\overline{\NN}$. Since we know that $GL_n$ and $\overline{\NN}$ act on $W$, we only need to show that
\begin{equation}
\label{conj}
\rho(g) \rho(\exp(\eta)) \rho(g^{-1}) = \rho( g \exp(\eta) g^{-1}),
\end{equation}
for all $g \in GL_n$ and $\eta \in \m \hL_+ / \m^N \hL_+$. 
We have 
$$\rho(g) \rho(\exp(\eta)) \rho(g^{-1})  = \rho(g) \exp(\rho(\eta)) \rho(g^{-1}) = \exp\left( \rho(g) \rho(\eta) \rho(g^{-1}) \right).$$ 
Let $g = \exp(u)$, where $u$ is a nilpotent element in $sl_n$. Then 
\begin{align*}
\rho(g) \rho(\eta) \rho(g^{-1}) &= \exp( \rho(u)) \rho(\eta) \exp(- \rho(u)) = \exp \left( \ad \rho(u) \right) \rho(\eta) \\
&= \rho( \exp( \ad u) \eta ) = \rho( g \exp(\eta) g^{-1} ).
\end{align*}
In a similar way, we can see that this equality is also valid for the scalar matrices in $GL_n$. Since the desired equality holds for the exponentials of the nilpotent elements, which, together with scalar matrices, generate $GL_n$, relation (\ref{conj}) holds for all $g \in GL_n$. This shows that we have an action of $GL_n \ltimes \overline{\NN}$ on $W$. Finally, if we postulate that the subgroup corresponding to $\m^N \hL_+$ acts on $W$ trivially, we get the action of the group $\Aut K[[X_1, \ldots, X_n]]$. 
\end{proof}

\section{Jets}

Passing from $AV$-module theory on an affine variety to a sheaf-based theory on non-affine varieties will require taking completions of our algebras. This is done using the notion of jets.

From now on, let $X$ be a smooth quasiprojective variety of dimension $n$ with the sheaves $\OO$ of functions,
$\V$ of vector fields, and $\D$ of differential operators.

In order to perform a local analysis of the sheaves we are interested in, we will use {\it \'etale charts}.

 \begin{definition}
An affine open subset $U \subset X$ is called an \'etale chart if there exist functions
$x_1, \ldots, x_n \in A = \A(U)$ such that

(1) the set $\{x_1, \ldots, x_n\}$ is algebraically independent, that is \break
$K[x_1, \ldots, x_n] \subset A$,

(2) every $f \in A$ is algebraic over $K[x_1, \ldots, x_n]$,

(3) derivations $\pd{x_1}, \ldots, \pd{x_n}$ of  $K[x_1, \ldots, x_n]$ extend to derivations of $A$. 
\end{definition}

We will call such $(x_1, \ldots, x_n)$ \emph{uniformizing parameters} on $U$. Since $A$ is algebraic over 
$K[x_1, \ldots, x_n]$, an extension of $\pd{x_i}$ to $A$ is unique. Vector fields  $\pd{x_1}, \ldots, \pd{x_n}$ commute.

\begin{lemma} (\cite[Theorem III.6.1]{Mu}, \cite{BF})
\label{Mumford}
Let $U$ be an \'etale chart of $X$ with uniformizing parameters $(x_1, \ldots, x_n)$.
Let  $A = \A(U)$, $V = \V(U)$, $D = \D(U)$. Then

(1)
$$ V = \mathop\bigoplus\limits_{i=1}^n A \pd{x_i},$$

(2)
$$ D = \mathop\bigoplus\limits_{k \in \Z_+^n} A \partial^k,$$
where for $k = (k_1, \ldots, k_n)$ we set $\partial^k = 
\left( \pd{x_1} \right)^{k_1} \ldots \left( \pd{x_n} \right)^{k_n}$.

(3)
$$\Omega^1 (U) =  \mathop\bigoplus\limits_{i=1}^n A d x_i$$
with the differential given by
$$df = \sum\limits_{i=1}^n \frac {\partial f}{\partial x_i} dx_i  \text{ \ for \ } f \in A.$$

(4) The map $(x_1, \ldots, x_n): \ U \rightarrow \Aff^n$ is \'etale.

\end{lemma}

Any smooth irreducible quasi-projective variety $X$ has an atlas $\{ U_i \}$ of \emph{\'etale charts}, $X = \mathop\bigcup\limits_i U_i$ (see e.g. \cite{Mu}, \cite{BF}).



Let $U$ be an affine open set, and let $A = \OO(U)$, $V = \V(U)$, $D = \D(U)$.

 Let $\Delta$ be the kernel of the multiplication map $A \otimes_K A \rightarrow A$. The algebra $J = A {\widehat\otimes}_K A$ of jets of functions is defined on affine open sets $U$ as a completion of  $A \otimes_K A$:
$$J = \varprojlim\limits_m (A \otimes_K A) / \Delta^m.$$
This yields the sheaf of jets of functions $\J$ with $\J(U) = J$.
  
\begin{lemma}[\cite{BI}]
\label{locjet}
Let $U$ be an \'etale chart of $X$ with uniformizing parameters $(x_1, \ldots, x_n)$, and let $A = \OO(U)$.
Then 
$${\it (1)} {\hskip 3.5cm} A {\widehat\otimes}_K A \cong A \otimes K[[X_1, \ldots, X_n]]. {\hskip 3.5cm}$$

\noindent
(2) The map $A {\otimes}_K A \rightarrow A \otimes K[[X_1, \ldots, X_n]]$, given by
$$g \otimes f \mapsto \sum_{s \in \Z^n_+} \frac{1}{s!} \, g \, \frac{\del^s f}{\del x^s}\, X^s,$$
extends to the completion and yields the above isomorphism of commutative algebras.

\noindent
(3) Under this isomorphism, the image of $ \delta(x_i) = 1 \otimes x_i - x_i \otimes 1$ is $X_i$.
\end{lemma}

Throughout this paper, we will use the following convention: when we write $f(x + X)$, we will understand the Taylor expansion in the second summand, that is,
$$f (x + X) = \sum_{s \in \Z^n_+} \frac{1}{s!}\, \frac{\del^s f}{\del x^s}\, X^s.$$
Thus, the map in (2) above, can be written as $g\otimes f \mapsto g(x) f(x + X)$. 

The following Taylor formula holds in  $A {\widehat\otimes}_K A$ \cite{BI}:
$$1 \otimes f = \sum_{s \in \Z^n_+} \frac{1}{s!} \left( \frac{\del^s f}{\del x^s} \otimes 1 \right) \delta(x)^s.$$
Under the isomorphism (1) in the above Lemma, this simply reads
$$1 \otimes f = f(x + X) = \sum_{s \in \Z^n_+} \frac{1}{s!} \frac{\del^s f}{\del x^s}\, X^s,$$

Now let us glue the local construction of Lemma \ref{locjet}  into the jet bundle. The bundle of jets of functions $\J$ is a bundle of commutative algebras with fiber $K[[X_1, \ldots, X_n]]$.
In each \'etale chart $\{ U, (x_1, \ldots, x_n) \}$ it trivializes as
$$\J (U) = \OO(U) \otimes K[[X_1, \ldots, X_n]].$$
Consider now two \'etale charts $\{ U_1, (x_1, \ldots, x_n) \}$, $\{ U_2, (y_1, \ldots, y_n) \}$ with the coordinate change on $U_1 \cap U_2$ given by
\begin{equation}
\label{coord}
x_i = G_i(y_1, \ldots, y_N), \ y_j = H_j (x_1, \ldots, x_N).
\end{equation}
Since each $y_j$ is only algebraic over $K[x_1, \ldots, x_n]$, and vice versa, we treat $G_i$'s and $H_j$'s as implicit functions. What is important for us is that their partial derivatives are well-defined.
Let $\J(U_1) = \OO(U_1) \otimes K[[X_1, \ldots, X_n]]$, $\J(U_2) = \OO(U_2) \otimes K[[Y_1, \ldots, Y_n]]$.
Then the transformation law for the sections of the jet bundle is:
$$X_i \mapsto G_i (y + Y) - G_i (y).$$

To justify this transformation law, we can perform the following computation:
$$X_i = \delta(x_i) = 1 \otimes x_i - x_i \otimes 1 = 1 \otimes G_i (y) - G_i(y) \otimes 1 = G_i (y+Y) - G_i (y).$$
It follows that this transformation law is compatible with the compositions of coordinate transformations, and hence invertible,
with the inverse $Y_j \mapsto H_j (x + X) - H_j (x)$.

Identifying variables $Y_i$ with $X_i$, we may view the above transformations as automorphisms of $K[[X_1, \ldots, X_n]]$
with coefficients in $\OO(U_1 \cap U_2)$.

This defines a map from the groupoid of coordinate transformations on $X$ to the bundle $\Aut \J$.

Over an affine open set $U$, the Lie algebra of jets of vector fields $JV$ is defined as a completion of $A \# V$:
$$ A \hh V =  A {\widehat\otimes}_K A \otimes_A V.$$
It is easy to check that $\Delta^m \otimes_A V$ are ideals in $A \# V$, thus this completion has a well-defined Lie bracket.
Taking the tensor product of sheaves $\JV = \J \otimes \Theta$, we obtain the sheaf of jets of vector fields.

We point out that $N$-differentiable $AV$-modules are precisely those annihilated by $\Delta^N \otimes_A V$. Thus any differentiable $AV$-module on an affine variety admits the action of the jets of vector fields.

The map $A \# V \rightarrow V$, given by multiplication, $g \# \eta \mapsto g\eta$, extends to the completion, since $\Delta \otimes_A V$ is in the kernel. Thus we have the map $ A \hh V \rightarrow V$, called the {\it anchor} map.
We will be particularly interested in the kernel of the anchor map.

Locally, in an \'etale chart, we define the completion of $A \# U(V)$ as
$$A \hh U(V) = \Us (A, A \hh V).$$
This glues into a quasicoherent sheaf $\AV$ of associative algebras. See \cite{BI} for details.

We define a sheaf of $AV$-modules as a sheaf of modules over the sheaf $\AV$.

%
%
%
%

\section{General construction of sheaves of $AV$-modules}

The key to understanding the structure of $AV$-modules is the following realization of the jets of vector fields and associative algebras $A \hh U(V)$ in \'etale charts:
\begin{theorem} [\cite{BI}] 
\label{iso}
Let $U$ be an \'etale chart with uniformizing parameters $(x_1, \ldots, x_n)$. Let $A$ be the algebra of polynomial functions on $U$, $V = \Der (A)$, and $D$ be the algebra of differential operators on $U$. Then 
$$A \hh V  \cong V \ltimes (A \otimes \hL_+),$$
and
$$A \hh U(V)  \cong D \otimes U (\hL_+).$$
\end{theorem}

The isomorphism between $A \hh V$ and $V \ltimes (A \otimes \hL_+)$ is given by the map
$$\varphi \left( g \# f \pd{x_i} \right) = gf \pd{x_i} + g(x)(f(x + X) - f(x)) \pd{X_i}.$$
Here when we write $f(x+X)$, we understand the Taylor expansion in the second argument, thus the above formula reads
$$\varphi \left( g \# f \pd{x_i} \right) = gf \pd{x_i} + \sum\limits_{k \in \Z_+^n \backslash \{ 0 \} } 
\frac{1}{k!} g \frac{\partial^k f}{\partial x^k} \otimes X^k  \pd{X_i}.$$
The inverse map is 
$$\psi \left( f \pd{x_i} \right) = f \# \pd{x_i},$$
$$\psi \left( g \otimes X^m \pd{X_i} \right) = (g \otimes 1) (1 \otimes x - x \otimes 1)^m \pd{x_i},$$
and extended to completions by continuity. Note that maps $\varphi$ and $\psi$ are homomorphisms of left $A$-modules.

The second part of the above theorem follows from its first claim by taking the strong enveloping algebras of both sides.

Note that in the present paper we use a different choice of signs from \cite{BIN, BI} when describing isomorphism $\psi$.

Define the sheaf of {\it virtual} jets of vector fields $\Lp$ as the kernel of the anchor map $\JV \rightarrow \V$.

\begin{corollary}
\label{keranchor}
Over an \'etale chart $U$, virtual jets of vector fields are realized as 
$$ \Lp (U) = A \otimes \hL_+ .$$
\end{corollary}

Let $\g$ be a Lie algebra.

\begin{definition}
A $\g$-bundle on $X$ is a sheaf $\F$ of Lie algebras on $X$ such that there exists a cover of $X$ by an atlas of affine open sets, 
where for each open set $U$ in this atlas, $\F(U)$ is isomorphic to $\OO(U) \otimes \g$ as Lie algebra in the category of $\OO(U)$-modules.
\end{definition}

\begin{example} If we have a principal $G$-bundle on $X$, where $G$ integrates $\g$, then there is an associated \emph{adjoint} $\g$-bundle. Identifying a principal $GL_{r}$ bundle with a rank $r$ locally free sheaf $\F$, the corresponding adjoint bundle is $\operatorname{End}(\F)$, with the evident Lie algebra structure. The trivial $G$-bundle induces the $\g$-bundle $\OO\otimes\g$, which can also be defined without the assumption of a group $G$ integrating $\g$. \end{example}

The sheaf $\V$ of vector fields is a sheaf of Lie algebras, but it is not a $\g$-bundle since the Lie bracket of vector fields is not $\OO$-linear. For the same reason,  the jets of vector fields $\JV$ is not a $\g$-bundle either.

However, it follows from Corollary \ref{keranchor} that the sheaf $\Lp$ of virtual jets of vector fields is an $\hL_+$-bundle. 
This bundle will play an important role in the present paper.

The coordinate transformation law for this bundle was given in \cite{BI}. Consider two \'etale charts: 
$\left\{ U_1, (x_1, \ldots, x_n) \right\}$ and $\left\{ U_2, (y_1, \ldots, y_n) \right\}$. 
Suppose on the intersection $U_1 \cap U_2$ the change of coordinates is given by (\ref{coord}).



Let $\hL^X_+$ and $\hL^Y_+$ be subalgebras of derivations of $K[[X_1, \ldots, X_n]]$ and $K[[Y_1, \ldots, Y_n]]$ respectively,
defined as above.

On the intersection $U = U_1 \cap U_2$ we have the maps
\begin{equation}
\label{localL}
\OO( U_1 \cap U_2) \otimes \hL^Y_+ \ \mathop\leftrightarrows\limits_{\psi_2}^{\varphi_2} \ \Lp (U) \ 
\mathop\rightleftarrows\limits_{\psi_1}^{\varphi_1} \ \OO( U_1 \cap U_2) \otimes \hL^X_+.
\end{equation}

Applying the composition of maps $\varphi_2 \circ \psi_1$, we get coordinate the transformation law for the sheaf $\Lp$:
\begin{align}
\label{Ltrans}
g(X) \dd{X_p} &= g( 1\otimes x - x \otimes 1 ) 1 \otimes \pp{}{x_p} \nonumber\\ 
&= g( 1\otimes G(y) - G(y) \otimes 1 ) 1 \otimes \pp{H_q}{x_p} (G(y)) \pp{}{y_q} \nonumber\\ 
&= g\left( G(y+Y) - G(y) \right) \frac{\partial H_q}{\partial x_p} (G(y+Y)) \dd{Y_q},
\end{align}
where $g$ is a power series in $K[[X_1, \ldots, X_n]]$ without a constant term. 

 
 Likewise, we have isomorphisms of associative algebras
\begin{equation}
\label{localAV}
D \otimes U(\hL^Y_+) \ \mathop\leftrightarrows\limits_{\psi_2}^{\varphi_2} \ \AV (U) \ 
\mathop\rightleftarrows\limits_{\psi_1}^{\varphi_1} \ D \otimes U(\hL^X_+).
\end{equation}

Again, considering the composition $\varphi_2 \circ \psi_1$, we get a homomorphism $D \rightarrow D \otimes U(\hL^Y_+)$,
given on the generators by the formula:
\begin{equation}
\label{dtrans}
\frac{\partial}{\partial x_i} \mapsto 
\sum\limits_{j=1}^n \frac{\partial H_j}{\partial x_i} (G(y)) \frac{\partial}{\partial y_j}
+ \left(   \frac{\partial H_j}{\partial x_i} (G(y+Y)) -  \frac{\partial H_j}{\partial x_i} (G(y)) \right) \frac{\partial}{\partial Y_j}.
\end{equation}
  
 For a $\g$-bundle $\F$, an $\F$-module is a sheaf $\M$ on $X$ with an $\OO$-linear Lie algebra morphism of sheaves $\F \rightarrow \End \M$.

%

Let $\F$ be a $\Lp$-module on $X$. We can use the action of $\Lp$ to define a charged $\D$-module structure on $\F$.

\begin{definition}
\label{deform}
 We call an $\Lp$-module $\F$ an $\Lp$-charged $\D$-module if for each \'etale chart $U$ with uniformizing parameters 
$(x_1, \ldots, x_n)$ we have a $\D(U)$-module structure on $\F(U)$ such that

(1) The actions of $\D(U)$ and $\Lp(U)$ are compatible in the following way:
$$\left[ \pd{x_i} , f \otimes g(X) \pd{X_j} \right] = \frac{\partial f}{\partial x_i} \otimes g(X) \pd{X_j},$$
$$ f_1 \left( f_2  \otimes g(X) \pd{X_j} \right) = f_1 f_2  \otimes g(X) \pd{X_j} .$$

\

(2) On the intersection of two \'etale charts $U_1$, $U_2$ 
the coordinate transformation for the action of differential operators is given by (\ref{dtrans}).
\end{definition}

\

\begin{lemma}
\label{inv}
 Condition 
$$\left[ \pd{x_i} , \ f \otimes g(X) \pd{X_j} \right] = \frac{\partial f}{\partial x_i} \otimes g(X) \pd{X_j},$$
is invariant under the coordinate change.
\end{lemma}
\begin{proof}
Let $U_1$ and $U_2$ be two \'etale charts with the coordinate change as above. Let us assume that in $U_2$ the condition
$$\left[ \pd{y_j} , f \otimes h(Y) \pd{Y_k} \right] = \frac{\partial f}{\partial y_j} \otimes h(Y) \pd{Y_k}$$
holds for any power series $h \in K[[Y_1,\ldots, Y_n]]$ with a zero constant term. Let us prove the analogous relation in chart $U_1$. 
We have
\begin{align*}
&\left[ \pd{x_i} , \ f(x) \otimes g(X) \pd{X_p} \right] \\
&= \bigg[ \frac{\partial H_j}{\partial x_i} (G(y)) \frac{\partial}{\partial y_j}
+ \left(   \frac{\partial H_j}{\partial x_i} (G(y+Y)) -  \frac{\partial H_j}{\partial x_i} (G(y)) \right) \frac{\partial}{\partial Y_j}, \\
& {\hskip 1cm} f\left( G(y) \right) g\left( G(y+Y) - G(y) \right) \frac{\partial H_q}{\partial x_p} (G(y+Y))
\dd{Y_q} \bigg] \\
&= \frac{\del f}{\del x_i}  \otimes g(X) \pd{X_p} + f  \left[ \pd{x_i} , \, 1 \otimes g(X) \pd{X_p} \right].
\end{align*}
Thus it is sufficient to prove that 
\begin{equation*}
\left[ \pd{x_i} , \ 1 \otimes g(X) \pd{X_p} \right] = 0.
\end{equation*}
Let us establish this equality. 
\begin{align*}
& \bigg[ \frac{\partial H_j}{\partial x_i} (G(y)) \frac{\partial}{\partial y_j}
+ \left(   \frac{\partial H_j}{\partial x_i} (G(y+Y)) -  \frac{\partial H_j}{\partial x_i} (G(y)) \right) \frac{\partial}{\partial Y_j}, \\
& {\hskip 1cm} g\left( G(y+Y) - G(y) \right) \frac{\partial H_q}{\partial x_p} (G(y+Y))
\dd{Y_q} \bigg] \\
& = \frac{\partial H_j}{\partial x_i} (G(y)) \frac{\del g}{\del x_\ell} \left( G(y+Y) - G(y) \right) 
\left( \frac{\del G_\ell}{\del y_j} (y+Y) - \frac{\del G_\ell}{\del y_j} (y) \right) \\
& {\hskip 7cm} \times \frac{\partial H_q}{\partial x_p} (G(y+Y))
\dd{Y_q} \\
&+ \frac{\partial H_j}{\partial x_i} (G(y)) g\left( G(y+Y) - G(y) \right)  \frac{\del^2 H_q}{\del x_p \del x_\ell} (G(y+Y))
\frac{\del G_\ell}{\del y_j} (y+Y)
\dd{Y_q} \\
&+ \left(   \frac{\partial H_j}{\partial x_i} (G(y+Y)) -  \frac{\partial H_j}{\partial x_i} (G(y)) \right)
 \frac{\del g}{\del x_\ell} \left( G(y+Y) - G(y) \right)  \\
&  {\hskip 6cm} \times \frac{\del G_\ell}{\del y_j} (y+Y) \frac{\partial H_q}{\partial x_p} (G(y+Y)) \dd{Y_q} \\
&+  \left(   \frac{\partial H_j}{\partial x_i} (G(y+Y)) -  \frac{\partial H_j}{\partial x_i} (G(y)) \right)
g\left( G(y+Y) - G(y) \right) \\
& {\hskip 6cm} \times \frac{\del^2 H_q}{\del x_p \del x_\ell} (G(y+Y))
\frac{\del G_\ell}{\del y_j} (y+Y)
\dd{Y_q} \\
&- g\left( G(y+Y) - G(y) \right) \frac{\partial H_q}{\partial x_p} (G(y+Y)) \\
& {\hskip 5cm}  \times \frac{\del^2 H_j}{\del x_i \del x_\ell} (G(y+Y)) \frac{\del G_\ell}{\del y_q} (y+Y)
\dd{Y_j}. \\
\end{align*}
Since $H$ and $G$ are inverses of each other, we have 
\begin{equation*}
\label{jac}
\frac{\partial H_j}{\partial x_i} (G(y)) \frac{\del G_\ell}{\del y_q} (y) = \delta_{i \ell}
\text{\ and \ }
\frac{\partial H_j}{\partial x_i} (G(y+Y)) \frac{\del G_\ell}{\del y_q} (y+Y) = \delta_{i \ell} .
\end{equation*}
Applying these relations we will see that all terms cancel out and we get zero.
\end{proof}

Our main result is a consequence of local isomorphisms (\ref{localL}):
\begin{theorem}
\label{main}
 An $\Lp$-charged sheaf $\M$ of $\D$-modules has the structure of a sheaf of $AV$-modules with the following action 
of vector fields in an \'etale chart:

$$\rho\left( f \frac{\partial}{\partial x_i} \right)  = f \frac{\partial}{\partial x_i}  + ( f(x + X) - f(x) )\frac{\partial}{\partial X_i}.$$
\end{theorem}

\begin{proof}
We need to show that $\M$ is a module over the sheaf $\AV$. Locally, in an \'etale chart $U$, \ $\AV(U) \cong \D(U) \otimes U(\hL_+)$.
By definition, $\M(U)$ admits commuting actions of $\OO(U) \otimes \hL_+$ and of $\D(U)$. Thus, it is a module for $\AV(U)$.
On the intersection $U_1 \cap U_2$ of two \'etale charts, the actions of $\AV(U_1)$ and $\AV(U_2)$ agree due to transformation laws (\ref{Ltrans}) and (\ref{dtrans}). 
\end{proof}

\section{Sheaves of jet modules}
\label{shjetmod}

In this section, we would like to generalize the sheaves of tensor modules and construct the sheaves of jet modules.

Fix a finite-dimensional representation $(W, \rho)$ for the Lie algebra $\hL_+$, for which the action of $L_0 \cong gl_n$ integrates to a rational $GL_n$-module. By Lemma \ref{Liecor}, the module $W$ admits the action of the group $\Aut K[[X_1, \ldots, X_n]]$, acting via its quotient $GL_n$. We will also denote this representation as $\rho$.

We would like to define the sheaf of $AV$-modules $\J^W$. Locally, in an \'etale chart $U$, this sheaf trivializes:
$$\J^W (U) = \OO(U) \otimes W.$$

Now consider two \'etale charts $\{ U_1, (x_1, \ldots, x_n) \}$, $\{ U_2, (y_1, \ldots, y_n) \}$ with the coordinate transformation
(\ref{coord}). Let $\varphi_G$ be the corresponding automorphism of $K[[X_1, \ldots, X_n]]$: 
$X_i \mapsto G_i (y + Y) - G_i (y)$.

The gluing transformation in the sheaf $\J^W$ is 
\begin{equation}
\label{gluing}
g(x) w \mapsto g(G(y)) \rho (\varphi_G) w \text{ \ for } w \in W.
\end{equation}

\begin{lemma}
\label{Lprop}
The vector bundle $\J^W$ is an $\Lp$-module.
\end{lemma}
\begin{proof}
First of all, we point out that transformation law (\ref{Ltrans}) is the conjugation by $\varphi_G$. Indeed,
\begin{align*}
\varphi_G g(X) \pp{}{X_p} \varphi_G^{-1} f(Y) &= \varphi_G g(X) \pp{}{X_p} f( H(x+X) - H(x)) \\
&= \varphi_G g(X) \pp{f}{y_q} (H(x+X) - H(x)) \pp{H_q}{x_p} (x + X) \\
&= g( G(y+Y) - G(y)) \pp{H_q}{x_p} (G(y+Y)) \pp{f}{y_q} (Y).
\end{align*}
Thus 
$$\varphi_G g(X) \pp{}{X_p} \varphi_G^{-1} =  g( G(y+Y) - G(y)) \pp{H_q}{x_p} (G(y+Y)) \pp{}{Y_q}.$$
Then
$$\varphi_G g(X) \pp{}{X_p}  =  g( G(y+Y) - G(y)) \pp{H_q}{x_p} (G(y+Y)) \pp{}{Y_q}\varphi_G.$$
Since this relation holds in representation $\rho$, the transformation law in $\J^W$ is compatible with the transformation law in $\Lp$, which implies the claim of the Lemma.
\end{proof}

\begin{theorem}
\label{bundle}
The vector bundle $\J^W$ is a sheaf of $AV$-modules. Locally, in an \'etale chart $U$ with uniformizing parameters 
$(x_1, \ldots, x_n)$, the action of vector fields is given by the formula:
\begin{align*}
 f \frac{\del}{\del x_i} (g \otimes w) &= f \, \pp{g}{x_i} \otimes w +
\sum_{s \in \Z^n_+ \backslash \{ 0 \}} \frac{1}{s!}\, g \, \frac{\del^s f}{\del x^s}\, \left( X^s \frac{\del}{\del X_i} \right) w \\
&= f \, \pp{g}{x_i} \otimes w + g(x) \left( f(x+X) - f(x) \right) \pp{}{X_i} w, 
\end{align*}
where $f, g \in A$, $w \in W$.
\end{theorem}

Note that by Lemma \ref{nilp}, the sum in the right-hand side is finite. The formula for this action for the Lie algebra of vector fields on a torus first appeared in \cite{B}.

\begin{proof}
Lemma \ref{Lprop} states that bundle $\J^W$ is an $\Lp$-module. We also have an obvious local $D$-module structure,
with $D(U)$ acting on the first tensor factor of $\OO(U) \otimes W$. We need to show that $\J^W$ is an $\Lp$-charged sheaf of $D$-modules. 

First, we would like to establish a commutation relation for $\pp{}{x_i} \circ \varphi_G^{-1}$:
\begin{align*}
\pp{}{x_i}  \varphi_G^{-1} & g(y) f(Y) = \pp{}{x_i} g(H(x)) f( H(x+X) - H(x)) \\ 
&= \pp{g}{y_j} (H(x)) \pp{H_j}{x_i}  f( H(x+X) - H(x)) \\ 
&+ g(H(x)) \pp{f}{y_j} \left( H(x+X) - H(x) \right) \left( \pp{H_j}{x_i}(x+X) - \pp{H_j}{x_i}(x) \right) \\
&=  \varphi_G^{-1} \bigg( \pp{H_j}{x_i} (G(y)) \pp{}{y_j} \\ 
& {\hskip 2cm}+ \left(  \pp{H_j}{x_i}(G(y+Y)) - \pp{H_j}{x_i}(G(y)) \right) \pp{}{Y_j} \bigg) g(y) f(Y).
\end{align*}
Thus, 
$$\pp{}{x_i}  \varphi_G^{-1} = \varphi_G^{-1} \bigg( \pp{H_j}{x_i} (G(y)) \pp{}{y_j} 
+ \left(  \pp{H_j}{x_i}(G(y+Y)) - \pp{H_j}{x_i}(G(y)) \right) \pp{}{Y_j} \bigg),$$
and the same relation holds in module $W$, but this is exactly what is required for $\J^W$ to be an $\Lp$-charged sheaf of $D$-modules. 

\end{proof}

If we apply the above jet module construction to a rational finite-dimensional $gl_n$-module (viewing it as a module for $\hL_+$ with a trivial action of $\m \hL_+$), we will recover the construction of a sheaf of tensor modules on $X$.

\begin{remark}
It is straightforward to see that tensor products and duality for jet modules match tensor products and duality of $\hL_+$-modules:
$$\J^{W_1} \otimes_\OO \J^{W_2} \cong \J^{W_1 \otimes W_2}, $$
$$\Hom_\OO (\J^W, \OO) \cong \J^{W^*}.$$ 
\end{remark}

\section{Realization of Rudakov modules with delta functions}
\label{Rud}

In a pioneering paper \cite{R} on the representation theory of Lie algebras of vector fields, Rudakov introduced and studied a class of modules for the vector fields on an affine space. This class of modules was generalized in \cite{BFN} to the case of arbitrary affine varieties. It was pointed out in \cite{BFN} that Rudakov modules are not just $V$-modules, but actually $AV$-modules.

Let us present the sheaf version of Rudakov modules. Fix a point $P \in X$ and a rational finite-dimensional $\hL_+$-module $W$.
Let $U$ be an \'etale chart of $X$ with uniformizing parameters $(x_1, \ldots, x_n)$, containing point $P$.

Let $\m_P$ be the maximal ideal in $A = \OO(U)$, corresponding to point $P$. Let $V = \Theta(U)$.

\begin{lemma} [\cite{BFN}]
\label{trunc}
There is an isomorphism of Lie algebras:
$$\m_P V / \m_P^{N+1} V \, \cong \, \hL_+ / \m^N \hL_+ .$$ 
The isomorphism is given by the expansion in local parameters:
$$ f \pp{}{x_i} \mapsto \sum_{0 < |s| < N} \frac{1}{s!} \, \frac{\del^s f}{\del x^s}(P) \, X^s \pp{}{X_i}.$$ 
\end{lemma}  

By Lemma \ref{nilp}, there exists $N \in \N$ such that $W$ is annihilated by $\m^N \hL_+$. The isomorphism of the previous Lemma allows us to view $W$ as a module for the Lie algebra $\m_P V$ with  $\m_P^{N+1} V$ acting trivially on $W$.
We also view $W$ as an $A$-module, with $f w = f(P) w$ for $f \in A$, $w \in W$.

Let us also state the following Lemma. Its proof may be given using the methods of \cite{BF}.
\begin{lemma}
\label{mVnilp}
Every non-zero ideal of $\m_P V$ contains  $\m_P^{N+1} V$ for some $N \in \N$.
\end{lemma}

\begin{corollary}
\label{mVann}
Let $W$ be a finite-dimensional representation of $\m_P V$. Then there exists $N \in \N$ such that 
$\m_P^{N+1} V$ annihilates $W$.
\end{corollary}
\begin{proof} Annihilator of a module is an ideal, and it must be non-zero since $\m_P V$ is infinite-dimensional and $W$ is finite-dimensional.
\end{proof}

Rudakov module $R^W_P$ is defined as the induced module
$$ R^W_P = \Ind_{A \# U(\m_P V)}^{A \# U(V)} W \cong K \left[ \pp{}{x_1}, \ldots, \pp{}{x_n} \right] \otimes W.$$

\begin{lemma}
\label{Rdiff}
Rudakov module $R^W_P$ is differentiable, it is annihilated by $\Delta^{N+1} \otimes_A V$.
\end{lemma}
\begin{proof} 
This proof is due to Henrique Rocha \cite{HR}. 

By Lemma 12 in \cite{BI}, the space $\Delta^{N+1} \otimes_A V$ is spanned by 
$$\left\{ (g \otimes 1) \delta (f_1) \delta (f_2) \ldots \delta (f_{N+1}) \eta \, | \, g, f_1, \ldots, f_{N+1} \in A, \eta \in V \right\}.$$
Since $A = K \cdot 1 \oplus \m_P$, we may assume that each $f_j$ is either 1, or belongs to $\m_P$. However, $\delta(1) = 0$.
Thus we may assume that all $f_j \in \m_P$. 

Let us show that $\Delta^{N+1} \otimes_A V$ annihilates $1 \otimes W$. 
\begin{align*}
(g \otimes 1) \delta (f_1) \delta (f_2) \ldots &\delta (f_{N+1}) \, \eta \, w \\
&= g \sum_{I \dot\cup J = \{1, \ldots, N+1 \} } (-1)^{|I|} \left( \prod_{i \in I} f_i \right) \left( \prod_{j \in J} f_j \cdot \eta \right) w.
 \end{align*}
The terms with $I \neq \varnothing$, $J \neq \varnothing$ vanish since in this case $\prod\limits_{j \in J} f_j \cdot \eta \in \m_P V$ and $\big( \prod\limits_{j \in J} f_j \cdot \eta \big) w  \in W$, while $\prod\limits_{i \in I} f_i$ annihilates $W$.
For the two remaining terms, $(f_1 \ldots f_{N+1} \eta) w = 0$ since $f_1 \ldots f_{N+1} \eta \in \m_P^{N+1} V$, while
$$(f_1 \ldots f_{N+1}) \, \eta \, w = (f_1 \ldots f_{N})\, \eta\,  f_{N+1} \, w +  (f_1 \ldots f_{N}) \,\eta(f_{N+1})\, w = 0.$$
%

The general case is then proved by induction on the degree of the monomial in $ \pp{}{x_1}, \ldots, \pp{}{x_n}$ in front of $w$.
To carry out the induction step, we will need the commutation relation:
\begin{align*}
\big[ \pp{}{x_i} , \, (g \otimes 1) \delta (f_1) \delta (f_2) \ldots &\delta (f_{N+1}) \eta \big] 
= (\pp{g}{x_i}\otimes 1) \delta (f_1) \delta (f_2) \ldots \delta (f_{N+1}) \eta \\
&+ \sum_{k=1}^{N+1}  (g \otimes 1) \delta (f_1) \ldots \delta \left(\pp{f_k}{x_i}\right) \ldots \delta (f_{N+1}) \eta  \\
&+  (g \otimes 1) \delta (f_1) \delta (f_2) \ldots \delta (f_{N+1})  \left[ \pp{}{x_i}, \eta \right].
\end{align*}
Decomposing $\pp{f_k}{x_i}$ into $K \cdot 1 + \m_P$, we will be able to carry out the step of induction.
\end{proof}

\begin{corollary}
The Rudakov module $R^W_P$ admits the action of the completed algebra $A\widehat{\#} U(V)$.
\end{corollary}

Note that the Rudakov module is supported at the point $P$.

Consider now the $D$-module $\F_P$ of delta functions supported at a point $P \in X$. Let $U$ be an \'etale chart containing point $P$. Let $\delta_P$ be the generator of the evaluation $A$-module: $f \delta_P = f(P) \delta_P$. This induces a $D$-module
$$\F_P (U) =  K \left[ \pp{}{x_1}, \ldots, \pp{}{x_n} \right] \otimes \delta_P.$$ As a sheaf, $\F_P$ is also supported at $P$.

Let us construct a realization of Rudakov modules using delta functions. Let $K_\tr$ be a 1-dimensional $gl_n$-module, with the action given by the trace of a matrix. If we view it as a $GL_n$-module, the action is given by the determinant, and the corresponding module of tensor fields is the module $\Omega^n$ of top differential forms. 


\begin{theorem}
\label{real}
Let $W$ be a rational finite-dimensional module for $\hL_+$, and let $P \in X$.
We have an isomorphism of sheaves of $AV$-modules:
$$\R_P^W \cong \F_P \otimes_\OO \J^{W \otimes K_\tr}.$$
\end{theorem}
\begin{proof}
Since both sheaves are supported at $P$, it is sufficient to verify the isomorphism locally, in an \'etale chart $U$, containing point $P$. We have
$$\R_P^W (U) \cong \F_P(U) \otimes_\OO \J^{W \otimes K_\tr}(U) $$
as vector spaces, since each of them is isomorphic to $K \left[ \pp{}{x_1}, \ldots, \pp{}{x_n} \right] \otimes W$.
Let us show that this identification is an isomorphism of $AV$-modules. Since the commutation relations of the elements of $A$ and $V$ with
$\pp{}{x_i}$ are the same in both modules, it is sufficient to show that the actions of $A$ and $V$ on $W$ agree.

In both modules, $A$ acts on $W$ by evaluation at $P$. Since in a jet module $\pp{}{x_i}$ annihilates the space $1 \otimes W$, 
we see that in both modules $\pp{}{x_i}$ acts on $W$ freely. The final case to consider is the action of $f \pp{}{x_i}$ with $f  \in \m_P$  on $W$ and $W \otimes K_\tr$ respectively. 

In Rudakov module $f \pp{}{x_i}$ acts on $W$ via the action (see Lemma \ref{trunc}):
$$\sum_{0 < |s| \leq N} \frac{1}{s!} \, \frac{\del^s f}{\del x^s}(P) \, \rho \left( X^s \pp{}{X_i} \right).$$
 In $\F_P \otimes_\OO \J^{W \otimes K_\tr}$ the action is
\begin{align*}
f \pp{}{x_i} \delta_P \otimes w \otimes 1_\tr 
&= \left( - \pp{f}{x_i}(P) \delta_P \right)  \otimes w \otimes 1_\tr \\
&+ \delta_P \otimes \sum_{0 < |s| \leq N} \frac{1}{s!} \, \frac{\del^s f}{\del x^s}(P) \, \rho \left( X^s \pp{}{X_i} \right) w \otimes 1_\tr \\
&+  \delta_P \otimes w \otimes \sum_{k=1}^n \pp{f}{x_k}(P) \tr(E_{ki}) 1_\tr.
\end{align*}
The first and the last terms in the right-hand side cancel out, and we get that the two actions on $\R_P^W$ and 
on $\F_P \otimes_\OO \J^{W \otimes K_\tr}$ agree.
\end{proof}

\begin{remark} We note that this gives another proof that Rudakov modules are differentiable, as both of the tensor factors in the decomposition of the above theorem are themselves differentiable. \end{remark}

In \cite{BFN} the authors constructed a contravariant pairing 
$$R_P^W \times J^{W*} \rightarrow K,$$
where $W^*$ is the dual $\hL_+$-module. Let us interpret this result in light of Theorem \ref{real}.
The above invariant pairing is equivalent to the existence of a homomorphism of $V$-modules (but not as $A$-modules!)
$$R_P^W \otimes_A J^{W*} \rightarrow K.$$
By Theorem \ref{real}, 
$$  R_P^W \otimes_A J^{W*} \cong \F_P \otimes_A  \Omega^n \otimes_A J^W \otimes_A J^{W^*}.$$
Since $J^{W^*} = \Hom_A ( J^W, A)$ as $AV$-modules, there exists a homomorphism of $AV$-modules
$$J^{W^*} \otimes_A J^W \rightarrow A.$$
Also, by Theorem \ref{real}, we can realize the Rudakov module corresponding to the trivial one-dimensional $\hL_+$-module $K \cdot 1$ as a tensor product:
$$R_P^{K \cdot 1} \cong \F_P \otimes_A \Omega^n.$$
It is easy to see that for this Rudakov module, the space $V R_P^{K(1)} $ is a $V$-submodule of codimension 1, which gives us a 
homomorphism of $V$-modules:
$$ \F_P \otimes_A \Omega^n \rightarrow K.$$
Combining, we get a chain of homomorphisms of $V$-modules:
$$R_P^W \otimes_A J^{W*}  \cong  \F_P \otimes_A \Omega^n \otimes_A J^W \otimes_A J^{W^*} 
\rightarrow  \F_P \otimes_A \Omega^n \otimes_A A \rightarrow K.$$
 
\begin{remark} Recall that by Kashiwara's lemma, $D$-modules supported on a subvariety $Z \subset X$ are precisely those pushed forward from $Z$. In particular, we have the $D$-module of delta functions along $Z$, here denoted $\F_{Z}$. Taking the tensor product of these with jet-type $AV$ modules produces many more $AV$ modules generalizing Rudakov's original construction. It remains an interesting problem to investigate functoriality of $AV$-modules under general maps, or even closed embeddings.
\end{remark}

\section{Holonomic $\AV$-modules}
\label{Holo}

Let $C$ be an associative algebra and $M$ be a $C$-module with a finite set $S$ of generators.
For a finite subset $B \subset C$, we define a filtration in $M$:
$$ F_0^B \subset  F_1^B \subset  F_2^B \subset \ldots,$$
where $F_0^B = \Span (S)$ and $F_{k+1}^B = \Span \left\{ F_k^B, \, B F_k^B \right\}$.

\begin{definition}
Gelfand-Kirillov dimension of $M$ is 
$$\text{GK}\dim (M) = \sup\limits_B \varlimsup\limits_{k \to \infty} \root{k}\of{\dim F_k^B}.$$
\end{definition}



\begin{definition}
A sheaf $\M$ of $\AV$-modules on $X$ is called holonomic if for every \'etale chart $U$, an $\AV(U)$-module $\M(U)$ 
is finitely generated and has Gelfand-Kirillov dimension $n = \dim X$. 
\end{definition}

It is easy to see that both jet modules and Rudakov modules are holonomic.

\medskip
{\bf Conjecture.} Every holonomic sheaf of $AV$-modules is differentiable.

It was proved in \cite{BR} that for a smooth affine variety every $AV$-module, which is finitely generated over $A$, is differentiable. This is a special case of the above conjecture. In Lemma \ref{Rdiff}, we proved that this conjecture also holds for Rudakov modules.

Let us finish this section with one more conjecture.

\medskip
{\bf Conjecture.} The Lie algebra of vector fields on a smooth affine variety is finitely generated as a Lie algebra.



\section{Examples of rank 2 bundles of $AV$-modules on $\Proj^1$.}

As an illustration of our methods, we will construct in this section two families of rank 2 bundles of $AV$-modules on $\Proj^1$.

In our first example, we begin with a family of 2-dimensional representations $W$ of $\hL_+$:
\begin{align*}
\rho_m \left( X \dd{X} \right) = &
\begin{pmatrix} m+1 & 0 \\ 0 & m \end{pmatrix}, \ \ 
\rho_m \left( X^2 \dd{X} \right) = 
\begin{pmatrix} 0 & 1 \\ 0 & 0 \end{pmatrix}, \\  
&\rho_m \left( X^k \dd{X} \right) = 0 \text{ \ for \ } k \geq 3.
\end{align*}

This representation is rational whenever $m \in \Z$, and integrates to the following representation of $\Aut K[[X]]$:
for $\varphi (X) = a X + b X^2 + $ {\it higher order terms},
$$\rho_m (\varphi) = a^m \begin{pmatrix} a & ba^{-1} \\ 0 & 1 \end{pmatrix}.$$

As usual, we cover $\Proj^1$ with two copies of $\Aff^1$ with the coordinate transformation between the charts given by 
$y = -x^{-1}$ (the negative sign will be more convenient for our computations). In our earlier notations, $H(x) = -x^{-1}$, $G(y) = -y^{-1}$, $\frac{dH}{dx} (G(y)) = y^2$.

The automorphism of $K[[X]]$, corresponding to this change of variables is
$$\varphi_G (X) = G(y+Y) - G(y) = y^{-2} Y - y^{-3} Y^2 + y^{-4} Y^3 + \ldots$$
and 
$$\rho_m (\varphi_G) = y^{-2m} \begin{pmatrix} y^{-2} & -y^{-1} \\ 0 & 1 \end{pmatrix}.$$

Let us denote the bases of the trivializations of the vector bundle $\J^W$ in $x$-chart by $\{ e_1^x, e_2^x \}$, and in $y$-chart by 
$\{ e_1^y, e_2^y \}$. Then by (\ref{gluing}), we have the following coordinate transformation in the bundle:
\begin{align*}
&e_1^x = y^{-2m-2} e_1^y, \\
&e_2^x = y^{-2m} e_2^y - y^{-2m-1} e_1^y.
\end{align*} 

Using (\ref{Ltrans}), we have the following transformations in the bundle $\Lp$ of virtual jets of vector fields on $\Proj^1$:
\begin{align*}
X \dd{X} &= Y\dd{Y} + y^{-1} \, Y^2\dd{Y}, \\
X^2 \dd{X} &=   y^{-2} \, Y^2\dd{Y}, \\
X^3 \dd{X} &=   y^{-4} \, Y^3\dd{Y} \,+ \text{\it higher order terms}. 
\end{align*}

It follows from Theorem \ref{bundle} (and could be easily verified directly), that coordinate transformation laws in $\Lp$ and $\J^W$ are compatible, and that $\J^W$ is, in fact, an $\Lp$-bundle. 

Bundle $\J^W$ is also an $\Lp$-charged sheaf of $D$-modules. In each chart, $D$-module structure is given by the action on functions, and the transformation law for $\pp{}{x}$ is given by (\ref{dtrans}):
$$\pp{}{x} = y^2 \pp{}{y} + 2y\, Y\dd{Y} + Y^2 \dd{Y},$$
and we could easily verify that this relation holds in $\J^W$.

The $AV$-module structure on $\J^W$ can be written explicitly as follows:
\begin{align*}
&f(x) \dd{x} \cdot g(x) e_1^x = f \frac{dg}{dx} \,e_1^x + (m+1) \,g \frac{df}{dx} \,e_1^x, \\
&f(x) \dd{x} \cdot g(x) e_2^x = f \frac{dg}{dx} \,e_2^x + m \,g \frac{df}{dx} \,e_2^x + \frac{1}{2}\, g \frac{d^2f}{dx^2} \,e_1^x, \\
\end{align*}
and analogously in $y$-chart.

Even though the representation we used here is rational only when $m \in \Z$, the above formulas yield a well-defined bundle for all $m \in \frac{1}{2}\Z$.

\

Our second example is based on another 2-dimensional representation of $\hL_+$:
\begin{align*}
&\sigma_m \left( X \dd{X} \right) = 
\begin{pmatrix} m+2 & 0 \\ 0 & m \end{pmatrix},  \ \ 
\sigma_m \left( X^2 \dd{X} \right) = 0 \\
&\sigma_m \left( X^3 \dd{X} \right) =
\begin{pmatrix} 0 & 1 \\ 0 & 0 \end{pmatrix},  \ \ \ \ \ \ \ \ 
\sigma_m \left( X^k \dd{X} \right) = 0 \text{ \ for \ } k \geq 4.
\end{align*}

When $m \in \Z$, this representation integrates to a rational representation of $\Aut K[[X]]$. For $\varphi(X) = a X + b X^2 + c X^3 +$ {\it higher order terms},
$$\sigma_m (\varphi) = a^m \begin{pmatrix} a^2 & ca^{-1} - b^2 a^{-2} \\ 0 & 1 \end{pmatrix}.$$
This gives us the coordinate transformation law in the bundle $\J^W$:
\begin{align*}
&e_1^x = y^{-2m-4} e_1^y, \\
&e_2^x = y^{-2m} e_2^y.
\end{align*}   

By Theorem \ref{bundle}, $\J^W$ is an $\Lp$-charged sheaf of $D$-modules, and is a bundle of $AV$-modules, with the following
explicit formulas for the action:
\begin{align*}
&f(x) \dd{x} \cdot g(x) e_1^x = f \frac{dg}{dx} \,e_1^x + (m+2) \,g \frac{df}{dx} \,e_1^x, \\
&f(x) \dd{x} \cdot g(x) e_2^x = f \frac{dg}{dx} \,e_2^x + m \,g \frac{df}{dx} \,e_2^x + \frac{1}{6}\, g \frac{d^3f}{dx^3} \,e_1^x.\\
\end{align*}

Again, this bundle is well-defined for all $m \in \frac{1}{2} \Z$.

\end{document}